\documentclass{amsart}
\usepackage{amssymb,amsmath,amscd,xy,graphicx,textcomp}
\newtheorem{theorem}{Theorem}[section]

\newtheorem{corollary}[theorem]{Corollary}
\newtheorem{definition}[theorem]{Definition}

\newtheorem{proposition}[theorem]{Proposition}

\xyoption{arrow}

\xyoption{matrix}

\def\C{\mathbb{C}}
\def\R{\mathbb{R}}
\def\Z{\mathbb{Z}}

\def\tree{\mathcal{T}}
\def\Z{\mathbb{Z}}
\def\trop{\mathbb{T}}

\title{Dissimilarity maps on trees and the representation theory of $GL_n(\C)$}
\author{Christopher Manon}
\thanks{This work was supported by the NSF fellowship DMS-0902710}

\begin{document}

\begin{abstract}
We revisit the representation theory in type $A$used previously to establish that the dissimilarity vectors of phylogenetic trees are points on the tropical Grassmannian variety.
We use a different version of this construction to show that the space of phylogenetic trees $K_n$ maps to the tropical varieties of every flag variety of $GL_n(\C).$  Using this map, we interpret the tropicalization of the semistandard tableaux basis of an irreducible representation of $GL_n(\C)$ as combinatorial invariants of phylogenetic trees.
\end{abstract}

\maketitle

\smallskip
\section{introduction}

The space $K_n$ of phylogenetic trees with $n$ leaves was introduced by Billera, Holmes and Vogtman in \cite{BHV} to give a geometric context to phylogenetic algorithms from mathematical biology.   As a space, $K_n$ is a fan, connected in codimension $1,$ with one maximal cone for each trivalent tree with $n$ leaves.

\begin{figure}[htbp]
\centering
\includegraphics[scale = 0.35]{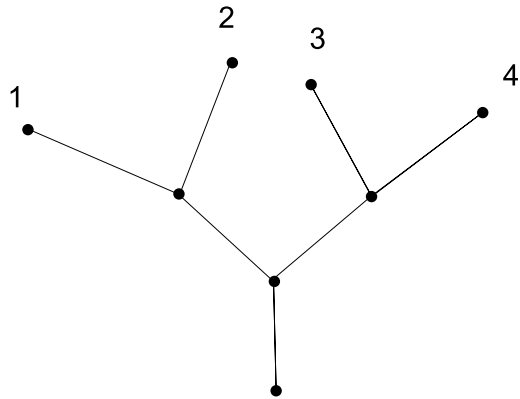}
\caption{a trivalent tree with $5$ leaves}
\label{fig:1}
\end{figure}

 A general point $\tree \in K_n$ is a tree with $n$ leaves with a specified metric, which is defined by an assignment of a non-negative real number $\ell(e)$ to each edge $e \in E(\tree).$  This defines a discrete metric on the set of leaves of $\tree.$   The pairwise distances between the leaves form a vector $d^2(\tree)$ of length $\binom{n}{2},$ which completely determines $\tree$. 

\begin{theorem}
If $d^2(\tree) = d^2(\tree')$ then $\tree = \tree'.$
\end{theorem}

This vector is known as the $2-$dissimilarity vector of $\tree.$
Not all discrete metrics come about in this way, so it is useful to have a theorem 
which classifies those which come from $2-$dissimilarity vectors of trees, this is where tropical geometry enters the picture.  Let $\trop = \R \cup\{-\infty\}$ denote the tropical real line.  For a polynomial $f = \sum C_{\vec{m}}x^{\vec{m}},$ the tropicalization is the following partial linear form. 

\begin{equation}
T(f) = max\{\ldots,  \sum m_i x_i, \ldots\}\\
\end{equation}

\noindent
The tropical variety $tr(f)$ associated to $f$ is the set of all points $\vec{t} \in \trop^N$
which make at least two linear forms in the expression $T(f)$ maximum.  For a polynomial ideal
$I,$ the tropical variety $tr(I)$ is defined as an intersection.

\begin{equation}
tr(I) = \bigcap_{f \in I} tr(f)\\ 
\end{equation}

The following theorem of Speyer and Sturmfels \cite{SpSt} classifies those discrete metrics which
come from $2$-dissimilarity vectors of trees in $K_n$.

\begin{theorem}
A vector $v \in \R^{\binom{n}{2}}$ is the $2-$dissimilarity vector of a tree $\tree \in K_n$
if and only if it is a point on the tropical variety defined by the Pl\"ucker embedding of the Grassmannian $tr(I_{2, n}).$   In particular, it is necessary and sufficient that for any distinct indicies $i, j, k, \ell,$ two of the following expressions must be equal and larger than the third. 

$$
v_{ij} + v_{k\ell}, v_{ik} + v_{j\ell}, v_{i\ell} + v_{jk}
$$
\end{theorem}

\noindent
In particular, the $2-$dissimilarity vectors define a map, which is $1-1$ and onto.  

\begin{equation}
d^2: K_n \to tr(I_{2, n})\\
\end{equation}

This lead  Pachter and Speyer \cite{PS} to consider the $m-$dissimilarity vectors $d^m(\tree)$
of metric trees, which are defined in a similar manner. 

\begin{definition}
Let $\sigma \subset [n]$ be a set of $m$ indices.  For a tree $\tree \in K_n$ define
$d_{\sigma}(\tree)$ to be the sum of the lengths of all edges which appear in the convex hull of the indices $\sigma.$   Define $d^m(\tree)$ to be the $\binom{n}{m}$ vector with entries $d_{\sigma}(\tree).$
\end{definition}

\begin{figure}[htbp]
\centering
\includegraphics[scale = 0.35]{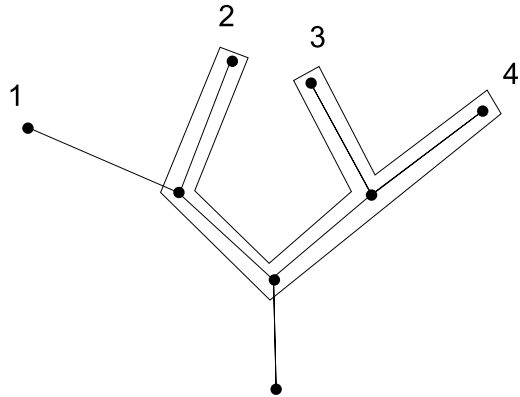}
\caption{The convex hull of three leaves}
\label{fig:2}
\end{figure}

Speyer, Pachter and later Cools conjectured a relationship between $d^m(\tree)$ and the higher tropical Grassmannian varieties.
Speyer and Pachter also showed that the $m-$dissimilarity vector of an $n-$tree $\tree$ determines
$\tree$ if $n \geq 2m-1.$

\begin{theorem}\label{dissimold}
The point $d^m(\tree)$ lies on the tropical Grassmannian $tr(I_{m, n}),$ in particular 
$d^m(\tree)$ is a solution to the tropicalization $tr(f)$ of each polynomial from the $(m, n)-$ Pl\"ucker ideal $f \in I_{m, n}.$ 
\end{theorem}

This theorem was proved by the author in \cite{M1} and Giraldo in \cite{G} with notably different techniques.  Our solution linked the combinatorics of the $m-$dissimilarity vector to the structure of the $(m, n)-$Pl\"ucker algebra as a representation of the special linear group $SL_m(\C).$  We refer the reader to the book of Fulton and Harris for the basics of the representation theory of the special linear group.   Letting $\omega_1$ be the first fundamental weight of $SL_m(\C),$ and for a weight $\lambda$ let $V(\lambda)$ be the corresponding representation, we have the following expression. 

\begin{equation}
P_{m, n} = \bigoplus_{\vec{r} \in \Z_{\geq 0}} [V(r_1\omega_1^*)\otimes \ldots \otimes V(r_n\omega_1^*)]^{SL_m(\C)}\\
\end{equation}

In some sense this is not the natural presentation of the Pl\"ucker algebra as an algebra with representation-theoretic meaning. It arrives by the first fundamental theorem of invariant theory, but a more natural way to obtain the Pl\"ucker algebra is as the projective coordinate ring of $Gr_m(\C^n)$ with its structure as a $GL_n(\C)$ variety. 

\begin{equation}
P_{m, n} = \bigoplus_{N \in \Z_{\geq 0}} V(N\omega_m)\\
\end{equation}

The purpose of this note is to establish tropical properties of dissimilarity vectors of 
trees using this different representation theoretic point of view.  Along the way we will show that not only the Grassmannian varieties, but every variety with $GL_n(\C)$-symmetry carries 
a map from Billera-Vogtman-Holmes tree space to its tropical varieties.   We employ the same method used in \cite{M2}, the language of branching algebras and branching valuations.  Our methods are applicable to reductive groups of other types and we intend to work out the space analagous to $K_n$ for type $D$ in a forthcoming publication.

\subsection{tropical structure of dissimilarity vectors}

In \cite{M2} our method was to construct the Billera-Holmes-Vogtman space of phylogenetic trees $K_n$ as a subfan of the valuations on the Pl\"ucker algebra $P_{m, n}.$  We then employed the following theorem from tropical geometry, for a commutative algebra $A$ we denote the space of valuations on $A$ into the tropical line $\trop$ by $\mathbb{V}_{\trop}(A).$ 

\begin{theorem}\label{trop}
For any set $x_1, \ldots x_n \in A,$ and any valuation $v \in \mathbb{V}_{\trop}(A),$
the point $(v(x_1), \ldots, v(x_n)) \in \trop^n$ lies in the tropical variety of the ideal of forms $I$ which vanish on $x_1, \ldots, x_n$ in $A.$

\begin{equation}
v(\vec{x}) \in tr(I)\\
\end{equation}

\end{theorem}

 It was then shown that evaluating the valuation $v_{\tree}$ corresponding to a metric tree $\tree \in K_n$ on the Pl\"ucker generators $z_{\sigma} \in P_{m, n}$ yielded the $m-$dissimilarity vector of $\tree,$ 

\begin{equation}
v_{\tree}(z_{\sigma}) = d_{\sigma}(\tree)\\
\end{equation}

\noindent
The valuations $v_{\tree}$ were constructed for $d^m(\tree)$ as a special case of a general method to construct valuations on algebras which capture branching information for morphisms of reductive groups.  As a sum of invariants in $n-$fold tensor products of irreducible $SL_m(\C)$ representations, the Pl\"ucker algebra can be realized as a subalgebra of the full tensor algebra of $SL_m(\C).$

\begin{equation}
P_{m, n} \subset R(\delta_{n-1})\\
\end{equation}

\noindent
This algebra captures the branching data for $\delta_{n-1}: SL_m(\C) \to SL_m(\C)^{n-1},$
the diagonal embedding, and it has all $n-$fold tensor products of irreducible $SL_m(\C)$ representations as graded summands. 

\begin{equation}
R(\delta_{n-1}) = \bigoplus_{\vec{\lambda} \in \Delta^n} Hom_{SL_m(\C)}(V(\lambda_0), V(\lambda_1) \otimes \ldots \otimes V(\lambda_{n-1}))\\
\end{equation}

In general, there is a multigraded algebra $R(\phi)$ assigned to a morphism of reductive groups $\phi:H \to G$ which algebraically encodes the problem of restricting representations from $G$ to $H,$
we will discuss these algebras and their deformations in section \ref{br}.    In \cite{M2} and \cite{M1} it was shown that each factorization  $\phi = \phi_1 \circ \ldots \circ \phi_m$ of $\phi$ in the category of reductive groups yields a cone $B(\vec{\phi}) \subset \mathbb{V}_{\trop}(R(\phi)).$  When coupled with Theorem \ref{trop} above, this gives a method for producing complexes of points on the tropical varieties associated to $R(\phi)$ and its sub-algebras, which inherit the combinatorial structure of the category of reductive groups. 
In particular, the diagonal morphism $\delta_{n-1}: SL_m(\C) \to SL_m(\C)^{n-1}$ comes with a factorization for every tree on $n$ leaves, given by other diagonal morphisms, see \cite{M2}.

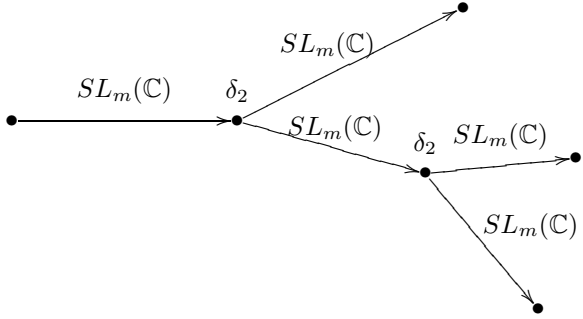
\begin{figure}[htbp]
\centering
\begin{xy}
(-15, 4)*{SL_m(\C)};
(12, 10)*{SL_m(\C)};
(35, -2)*{SL_m(\C)};
(39, -14)*{SL_m(\C)};
(13, -1)*{SL_m(\C)};
(0,4)*{\delta_2};
(25, -3 )*{\delta_2};
(0,0)*{\bullet} = "AA";
(30,15)*{\bullet} = "BB";
(45,-5)*{\bullet} = "CC";
(40,-25)*{\bullet} = "DD";
(-30,0)*{\bullet} = "EE";
(25,-7)*{\bullet} = "FF";
"BB"; "AA";**\dir{-}? >* \dir{>};
"AA"; "EE";**\dir{-}? >* \dir{>};
"FF"; "AA";**\dir{-}? >* \dir{>};
"DD"; "FF";**\dir{-}? >* \dir{>};
"CC"; "FF";**\dir{-}? >* \dir{>};
\end{xy}
\caption{a factorization of $\delta_3$}
\end{figure}

This gives a cone $C(\tree)$ for each such tree, and the valuations $v_{\tree}$ are 
special points in this cone, chosen for their desireable properties when evaluated on the Pl\"ucker generators. In this way the factorization properties of diagonal maps realize dissimilarity vectors of trees by valuations on the Pl\"ucker algebra.  This is the essence of the proof of Theorem \ref{dissimold} from \cite{M2}.  Notably this produces many more valuations then just those coming from $K_n,$ presumably these hide interesting combinatorial information about the trees $\tree,$ the Pl\"ucker algebra $P_{m, n}$ and other algebras inside the full tensor algebra. 

\subsection{factorization by general linear subgroups of $GL_n(\C)$}

 The Pl\"ucker algebra can also be constructed as a subalgebra of the branching algebra $R(1_G),$ associated to the identity morphism $1_G:1 \to GL_n(\C).$

\begin{equation}
R1_G) = \bigoplus_{\lambda \in \Delta} V(\lambda)\\
\end{equation}

\noindent
Multiplication in this algebra is computed via Cartan multiplication, $V(\lambda)\otimes V(\eta) \to V(\lambda + \eta).$  This is the coordinate ring of $GL_n(\C)/U_+,$ where $U_+$ is the subgroup of upper triangular unipotent matrices.  This is also the Cox ring of the full flag variety, in particular the projective coordinate ring of any flag variety $GL_n(\C)/P$, for any ample line bundle $\mathcal{L}(\lambda)$ sits inside $R(1_G)$ as a subalgebra.

\begin{equation}
R_{\lambda} = \bigoplus_{N \geq 0} V(N\lambda) \subset R(1_G)\\
\end{equation}

\noindent
In this way we obtain the Pl\"ucker algebra is a subalgebra of $R(1_G)$ as $P_{m, n} = R_{\omega_m}.$    This allows us to once again analyze some tropical geometry of the Grassmannian, and indeed any flag variety in type $A,$ using the branching construction on factorizations of $1 \to GL_n(\C).$  In particular we can obtain a cone of valuations in $\mathbb{V}_{\trop}(R_{\omega_m})$ for any nested sequence of subgroups of $GL_n(\C)$.  
Our main tool will be nested products of general linear groups.

 For a set of indices $\sigma \subset [n]$ of size $|\sigma| = k$ we define the subgroup $i_{\sigma}: GL_k(\C) \to GL_n(\C)$ to be the subcopy of $GL_k(\C)$ on the indices $\sigma.$  We represent this with the model matrix diagram below.

\[ \left( \begin{array}{ccccc}
a & b & 0 & c & 0\\
d & e & 0 & f & 0 \\
0 & 0 & X & 0 & Y\\
g & h & 0 & i & 0\\
0 & 0 & W & 0 & Z\\
\end{array} \right)\] 

Here an element of $GL_3(\C) \times GL_2(\C)$ on the indices 
$\{1, 2, 4\}$ and $\{3, 5\}$ respectively defines an element of $GL_5(\C).$
In this way, every $n+1$ tree $\tree$ defines a directed system of general linear groups,
where an edge $e \in E(\tree)$ defines the subgroup $i_{L(e)}(GL_k(\C)) \subset GL_n(\C)$
where $L(e)$ is the set of leaves $l$ such that the unique path between $l$ and $0$ in $\tree$ passes through $e.$

\begin{figure}[htbp]

\centering
\begin{xy}
(-15, 4)*{GL_3(\C)};
(12, 10)*{GL_1(\C)};
(35, -2)*{GL_1(\C)};
(39, -14)*{GL_1(\C)};
(13, -1)*{GL_2(\C)};
(0,0)*{\bullet} = "AA";
(30,15)*{\bullet} = "BB";
(45,-5)*{\bullet} = "CC";
(40,-25)*{\bullet} = "DD";
(-30,0)*{\bullet} = "EE";
(25,-7)*{\bullet} = "FF";
"AA"; "BB";**\dir{-}? >* \dir{>};
"EE"; "AA";**\dir{-}? >* \dir{>};
"AA"; "FF";**\dir{-}? >* \dir{>};
"FF"; "DD";**\dir{-}? >* \dir{>};
"FF"; "CC";**\dir{-}? >* \dir{>};
\end{xy}
\caption{a factorization of $1_{GL_3(\C)}$}
\end{figure}
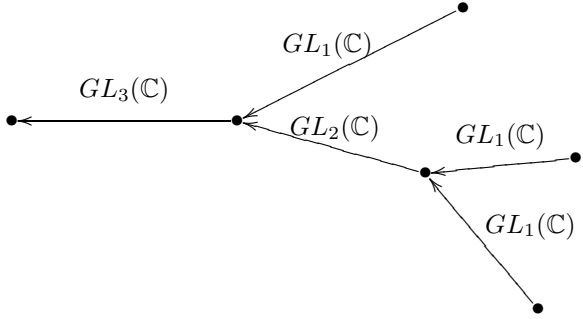

Each vertex $v \in V(\tree)$ then corresponds to a map $i_{L(e_1)}\times \ldots \times i_{L(e_j)}: GL_{k_1}(\C) \times \ldots \times GL_{k_j} \to GL_m(\C),$ where $k_i = |L(e_i)|,$ and $m = |L(e)|.$  In particular, the leaves of $\tree$ are assigned the corresponding copy of $\C^* \subset GL_n(\C)$ at the proper index in the diagonal.

Following the branching construction, we get a valuation on the algebra $R(1_G)$ by assigning a function $\rho,$ linear on the weights of the appropriate $GL_k(\C)$ to each edge of this tree. We need $\rho$ to act in such a way that every dominant weight receives a positive real value, and for every pair of dominant weights $\lambda, \eta,$ the weights appearing in $V(\lambda) \otimes V(\eta)$ receive weight less than or equal to $\rho(\lambda) + \rho(\eta) = \rho(\lambda + \eta).$  We choose the functional $\rho$ that counts the number of boxes in the top row of the Young Diagram of the associated weight $\lambda.$

\begin{figure}[htbp]
\centering
\includegraphics[scale = 0.55]{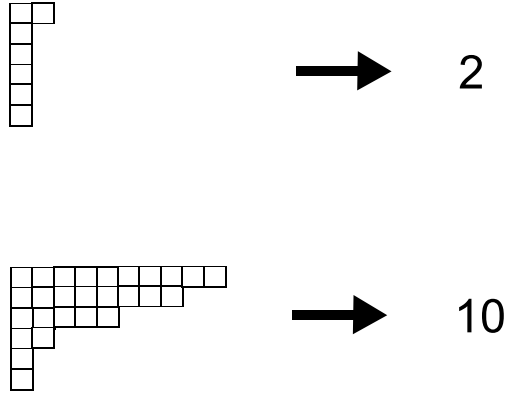}
\caption{weight of a Young tableaux}
\label{fig:6}
\end{figure}

\noindent
This particular functional has the nice property that $\rho(\omega_i) = 1$ for $i = 1, \ldots, n$
(recall that $\omega_n = det$).   In this way, a metric tree  $\tree$ with topological structure equal to the chosen structure defines a valuation $w_{\tree}$ on $R(1_G)$ by the assignment of $\ell(e)\rho$ to $e \in E(\tree).$

\subsection{branching diagrams}
In section \ref{br}  we will use this to prove the following proposition about the exterior forms, which form a basis of $V(\omega_i).$

\begin{theorem}\label{dissim}
Let $\tree$ be an $n+1$ tree, and $z_{\sigma}= z_{i_1}\wedge \ldots \wedge z_{i_m} \in \bigwedge^m(\C^n).$
Then $w_{\tree}(z_{\sigma}) = d_{0, \sigma}(\tree).$
\end{theorem}

Note that the vector associated to the Pl\'Ucker generators of $P_{m, n}$ by this method
is not the complete $m-$dissimilarity vector of $\tree,$ as each component always contains an index $0.$  We are not re-proving Theorem \ref{dissimold}, but we are showing that these specialized components of dissimilarity vectors are also solutions to tropical equations coming from Grassmannians.  Using the branching algebra $R(1_G)$ also offers enough flexibility to allow us to say something about general flag varieties, and even general algebras with a $GL_n(\C)$ action. 

Recall that a representation $V(\lambda)$ of $GL_n(\C)$ has a basis given by semistandard
fillings $T$ of the Young diagram associated to the weight $\lambda.$  Let $z_T$ be the basis member of $V(\lambda)$ associated to the filling $T.$   Each $T$ can be viewed as a collection of sets of indices, corresponding to the columns of $T.$

\begin{equation}
T = [\sigma_1, \ldots, \sigma_k]\\
\end{equation}

\noindent
Indeed, the element $z_T$ is the image of the tensor $z_{\sigma_1} \otimes \ldots \otimes z_{\sigma_k}$ under the unique map of representiontations.

\begin{equation}
\pi_{\lambda}:V(\omega_{|\sigma_1|})\otimes \ldots \otimes V(\omega_{|\sigma_k|}) \to V(\sum \omega_{|\sigma_i|}) = V(\lambda)\\
\end{equation}

\noindent
This map is the multiplication map in $R(1_G),$ so from general properties of valuations 
we get the following corollary.   We define $d_{0, T}(\tree) = \sum_{\sigma_i \in T} d_{0, \sigma_i}(\tree).$

\begin{corollary}\label{tab}
For any semistandard tableaux $T$ with associated element $z_T \in R(1_G)$ we have the following equation.

\begin{equation}
w_{\tree}(z_T) = d_{0, T}(\tree)\\
\end{equation} 

\end{corollary}

This gives a map from $K_{n+1}$ to the tropical varieties of any projective coordinate
ring of any line bundle over any flag variety of type $A,$ or indeed any algebra with
a standard tableaux generating set.   We also get the following generalization of Theorem \ref{dissim}

\begin{corollary}
Let $f \in I_{\lambda}$ be an element of the ideal which vanishes on the semistandard tableaux generators of the projective coordinate ring $R_{\lambda},$ associated to $\mathcal{L}(\lambda)$ of a flag variety $GL_n(\C)/P(\lambda).$ Then the tableaux dissimilarity components $d_{0, T_i},$ as $T_i$ runs over the basis of $V(\lambda),$ satisfy $T(f).$
\end{corollary}

%\begin{example}
%We will look at the dissimilarity coordinates of the tree below, given by the generators of the %algebra $R(1_G)$ for $GL_3(\C).$
%EXAMPLE HERE
%\end{example}

In \cite{M2} we showed a very general result on $G-$algebras, Theorem 3.5.  A consequence
of this is that any complex of valuations arrising from the branching valuation construction
on $R(1_G)$ also passes to a complex of valuations on an algebra $A$ with a rational $GL_n(\C)$ action.  

\begin{theorem}\label{dissimgen}
Let $A$ be a commutative algebra with a rational action by $GL_n(\C),$ then there is a map
$K_{n+1} \to \mathbb{V}_{\trop}(A).$
\end{theorem}

Just as each $R_{\lambda}$ defined its own set of tropical invariants of trees $\tree \in K_{n+1}$ in the form of the tropicalizations of semi-standard tableaux, any generating set $X \subset A$ of such an algebra will define a set of tropical invariants of $\tree$ which satisfy the tropicalized equations in the ideal defined by the presentation of $A$ by $X.$

\section{overview of branching valuations}\label{br}

Here we give an overview of the essential ideas used to proved Theorem \ref{dissim}.  
For more on branching valuations we direct the reader to \cite{M2} and \cite{M1}.  
For a map of reductive groups $\phi: H \to G$ we define a commutative algebra
over $\C.$

\begin{equation}
R(\phi) = [R_H \otimes R_G]^H\\
\end{equation}

\noindent
Here invariants are defined with respect to the action of $H$ on $R_H$
and $H$ on $R_G$ through $\phi.$  As a multigraded vector space, this algebra
can be formulated as follows, 

\begin{equation}
R(\phi) = \bigoplus_{\eta, \lambda \in \Delta_H\times \Delta_G} Hom_H(V(\eta), V(\lambda))\\
\end{equation}

\noindent
where the sum runs over all pairs of dominant $H$ and $G$ weights.  Now we factor
$\phi$ by a pair of morphisms in the category of reductive groups. 

$$
\begin{CD}
H @>\psi>> K @>\pi>> G\\
\end{CD}
$$

\noindent
From this factorization we may refine the multigrading presented above. 

\begin{equation}
Hom_H(V(\eta), V(\lambda)) = \bigoplus_{\tau \in \Delta_K}Hom_H(V(\eta), V(\tau)) \otimes Hom_K(V(\eta), V(\lambda))\\
\end{equation}

This is a formal consequence of the semisimplicity of the categories of finite dimensional representations
of reductive groups.  Notice that the components on the right hand side can be viewed as summands
in $R(\psi) \otimes R(\pi).$  We simplify notation a little and rename the above summands.

\begin{equation}\label{decompose}
W(\eta, \tau, \lambda) = Hom_H(V(\eta), V(\tau)) \otimes Hom_K(V(\eta), V(\lambda))\\
\end{equation}

\noindent
The following theorem can be found in \cite{M1} and \cite{M2}.

\begin{theorem}
The following holds under multiplication in the algebra $R(\phi).$
\begin{equation}
W(\eta_1, \tau_1, \lambda_1) \times W(\eta_2, \tau_2, \lambda_2) \subset \bigoplus_{\tau \leq \tau_1 + \tau_2} W(\eta_1 + \eta_2, \tau, \lambda_1 + \lambda_2)\\ 
\end{equation}

\noindent
Moreover, when projected onto the highest weight component on the right hand side, the resulting multiplication operation 
agrees with multiplication of the corresponding components in $R(\psi)\otimes R(\pi).$
\end{theorem}

The same construction can be made for any chain of group morphisms. 

$$
\begin{CD}
H @>\phi_1>> K_1 @>\phi_2>> \ldots @>\phi_m>> K_m @>\phi_{m+1}>> G\\
\end{CD}
$$

\noindent
This results in a multifiltration of $R(\phi)$ by tuples of dominant weights from $K_1, \ldots K_m.$
We can turn this into a valuation in a number of ways, each given by a choice of functional $\rho_i$ 
on the dominant weights for each $K_i$ which assigns the largest number to $\tau_1 + \tau_2$ in the sum 
in the theorem above.  There is a cone $B(\vec{\phi})$ of these functionals, its properties are discussed
in \cite{M1}.   The cones for different factorizations $\vec{\phi}$ fit together into a complex $K_{\phi}$ of valuations
on $R(\phi),$ this is also discussed in \cite{M1}.   The important point here is that this complex inherits combinatorial
properties naturally from the category of reductive groups.

In order to evaluate a valuation $v_{\vec{\rho}}$ on an element $f \in R(\phi)$ one must compute the branching diagram of $f$.   This means we must decompose $f$ into its homogenous components along
the multifiltration as in Equation \ref{decompose}.  This entails recording the dominant weights of $f$
considered as a vector in a representation of each group $K_i.$  Once this is done, the value of the valuation
is given by applying the functional $\vec{\rho}$ to the resulting tuple of dominant weights.
In section \ref{tableaux} below we do this for our tree valuations and members of the basis of $R_{GL_n(\C)}$ defined by the semistandard tableaux.

\section{dissimilarity vectors and semistandard tableaux}\label{tableaux}

In this section we compute branching diagrams for the basis $B(\lambda) \subset V(\lambda)$
of semistandard tableaux of an irreducible representation of $GL_n(\C)$ for the branching by given by an $n+1$ tree $\tree$ constructed in the introduction.   We then explain how to produce a valuation $w_{\tree}$ on $R(1_G)$ from a metric on $\tree,$ and compute this valuation on members of the basis $B(\lambda),$ this will prove Theorem \ref{dissim} and Corollary \ref{tab}.

We begin by computing the branching diagrams of the exterior forms $z_{\sigma},$ $|\sigma| = m,$ which constitute a basis of $\bigwedge^m(\C^n).$  This representation branches in a very simple way over the subgroup $GL_k(\C)\times GL_{n-k}(\C) \subset GL_n(\C)$ corresponding to the decomposition $\C^n = \C^k \oplus \C^{n-k}.$ We have the following. 

\begin{equation}
\bigwedge^m(\C^n) = \bigoplus_{i + j = m} \bigwedge^i(\C^k) \otimes \bigwedge^k(\C^{n-k})\\
\end{equation}

Let $I and J$ be the subsets of the index set $[n]$ defined by the above direct sum decomposition, then in terms of the exterior form basis of each space involved we have,

\begin{equation}
z_{\sigma} = z_{\sigma \cap I} \wedge z_{\sigma \cap J} \subset \bigwedge^{|\sigma \cap I|}(\C^k) \otimes \bigwedge^{|\sigma \cap J|}(\C^{n-k}).\\
\end{equation}

\noindent
The result of this step is an exterior form $z_{\sigma \cap I}$ on each new branch of the diagram, so the rest of the branching diagram is computed in the same way.  This observation proves the following. 

\begin{proposition}
For an $n+1$ tree $\tree,$ an edge $e \in E(\tree),$ and an exterior form $z_{\sigma} \in \bigwedge^m(\C^n),$ the branching diagram of $z_{\sigma}$ has a tableaux of shape 
$[1, \ldots, 1, 0, \ldots 0]$ at $e,$ where the number of $1'$s is equal to $|L(e) \cap \sigma|,$
and it is thought of as a dominant weight for $GL_{|L(e)|}(\C).$
\end{proposition}

\begin{figure}[htbp]
\centering
\includegraphics[scale = 0.65]{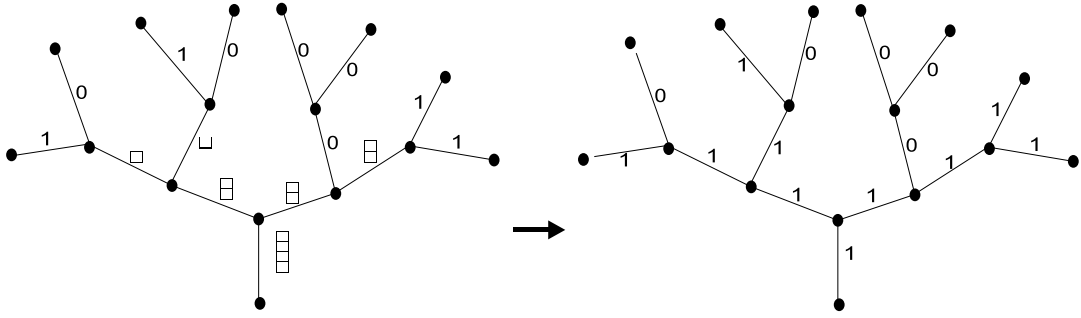}
\caption{trivalent tree}
\label{fig:8}
\end{figure}

For $T$ a semistandard tableaux, we can represent $z_T \in V(\lambda), |T| = \lambda$
as a tensor product of exterior forms.  Let $T = [\sigma_1, \ldots, \sigma_k],$ then there is a surjection $\pi_{\lambda}:V(\omega_{|\sigma_1|})\otimes \ldots \otimes V(\omega_{|\sigma_k|}) \to V(\lambda),$ such that the following holds. 

\begin{equation}
z_T = \pi_{\lambda}(z_{\sigma_1}\otimes \ldots \otimes z_{\sigma_k})\\
\end{equation}

\begin{corollary}\label{c}
For an $n+1$ tree $\tree,$ and  element $z_T$ the branching diagram of $z_T$ is the sum of the branching diagrams
for the $z_{\sigma_i},$ defined above. 
\end{corollary}

\begin{proof}
This follows from the general version of Theorem \ref{dissim} above.
\end{proof}

This corollary establishes enough for us to evaluate our branching valuations on the basis
$z_T \in B(\lambda) \subset V(\lambda).$   Once again following section \ref{br}, we now produce the valuation $w_{\tree}$ associated to a metric tree $\tree.$  From the previous section, it suffices to assign a coweight to each edge $e \in E(\tree).$

\begin{definition}
Let $\rho$ be the coweight which counts the number of boxes in the first row of the tableaux representing a weight $\lambda.$  The valuation $w_{\tree}$ is defined by assigning $\ell(e) \rho$ to the edge $e$, where $\ell(e)$ is the length of the edge $e \in E(\tree).$
\end{definition}

It is then clear that $w_{\tree}(z_{\sigma})$ gives the sum of the lengths of the edges $e$
which appear in the combinatorial convex hull of the indices $\sigma \cup \{0\},$ as these are the only edges with non-0 branching weight, and the coweight $\rho$ assigns each non-zero weight in this diagram a $1.$  This proves Theorem \ref{dissim}. 

%\begin{example}
%We can compute the dissimilarity coordinates given by the basis of the representation $V(1, %1,0)$ of $GL_3(\C).$

%EXAMPLE

%\end{example}

Even though we are taking only some of the components of the dissimilarity vectors of an $n+1$ tree $\tree,$ we can still recover the structure of $\tree$ from this information.  In fact, it is enough to just have the $v_{\tree}$ evaluations of the degree $1$ and degree $2$ Pl\"ucker coordinates.   As defined, this set already includes all of the $2$-components $d_{0,j},$ and it is simple to verify the following.

\begin{equation}
d_{i,j} = 2d_{0, i, j} - d_{0, i} - d_{0, j}\\
\end{equation}

\noindent
Since the $2$-disimilarity vector of a tree completely determines the metric structure, this shows that $w_{\tree}(z_i)$ and $w_{\tree}(z_i \wedge z_j)$ determine $\tree.$
Theorem \ref{dissim} above also implies that the rooted dissimilarity numbers $d_{0, \sigma}$ satisfy many combinatorial tropical equations. 

%\begin{example}
%RELATIONS FOR THE FULL FLAG VARIETY

%\end{example}

%\begin{example}
%RELATIONS WITH MIXED ENTRIES
%\end{example}

\bigskip
\noindent
Christopher Manon:\\
Department of Mathematics,\\ 
University of California, Berkeley,\\ 
Berkeley, CA 94720-3840 USA, 

\end{document}